\documentclass[a4paper, 11pt]{amsart}   
\usepackage{mathptmx, amssymb,amscd,latexsym, eulervm}   
\usepackage{amsmath}
\usepackage{amsthm}
\usepackage{mathdots}
\usepackage[pagebackref,colorlinks=true,linkcolor=blue,urlcolor=blue]{hyperref}
\usepackage{color}
\usepackage[onehalfspacing]{setspace}
\usepackage{tabularx}
\usepackage{amsfonts}
\usepackage{paralist}
\usepackage{aliascnt}
\usepackage[initials, lite]{amsrefs}
\usepackage{amscd}
\usepackage{blkarray}
\usepackage{mathbbol}
\usepackage{setspace}
\usepackage[inner=2.5cm,outer=2.5cm, bottom=3.2cm]{geometry}
\usepackage{tikz, tikz-cd}

\usepackage{tikz}
\usetikzlibrary{matrix}
\usetikzlibrary{arrows,calc}
\allowdisplaybreaks

\BibSpec{collection.article}{%
	+{}  {\PrintAuthors}                {author}
	+{,} { \textit}                     {title}
	+{.} { }                            {part}
	+{:} { \textit}                     {subtitle}
	+{,} { \PrintContributions}         {contribution}
	+{,} { \PrintConference}            {conference}
	+{}  {\PrintBook}                   {book}
	+{,} { }                            {booktitle}
	+{,} { }                            {series}
	+{, vol.} { }                            {volume}
	+{,} { }                            {publisher}
	+{,} { \PrintDateB}                 {date}
	+{,} { pp.~}                        {pages}
	+{,} { }                            {status}
	+{,} { \PrintDOI}                   {doi}
	+{,} { available at \eprint}        {eprint}
	+{}  { \parenthesize}               {language}
	+{}  { \PrintTranslation}           {translation}
	+{;} { \PrintReprint}               {reprint}
	+{.} { }                            {note}
	+{.} {}                             {transition}
	+{}  {\SentenceSpace \PrintReviews} {review}
}
\AtBeginDocument{%
	\def\MR#1{}
}

\newcommand{\kk}{\mathbb{k}}

\newcommand{\NN}{\normalfont\mathbb{N}}
\newcommand{\ZZ}{{\normalfont\mathbb{Z}}}

\newcommand{\mm}{{\normalfont\mathfrak{m}}}

\newcommand{\pp}{\mathfrak{p}}

\newcommand{\bb}{\mathfrak{b}}
\newcommand{\qqq}{\normalfont\mathfrak{q}}

\newcommand{\reg}{\normalfont\text{reg}}

\newcommand{\depth}{\normalfont\text{depth}}

\newcommand{\Tor}{\normalfont\text{Tor}}
\newcommand{\Ext}{\normalfont\text{Ext}}

\newcommand{\Coker}{\normalfont\text{Coker}}

\newcommand{\Hom}{\normalfont\text{Hom}}
\newcommand{\grHom}{{}^*\Hom}

\newcommand{\FF}{\mathcal{F}}
\newcommand{\HL}{\normalfont\text{H}_{\mm}}
\newcommand{\HH}{\normalfont\text{H}}

\newcommand{\iniTerm}{\normalfont\text{in}}

\newcommand{\Proj}{\normalfont\text{Proj}}
\newcommand{\Spec}{{\normalfont\text{Spec}}}

\newtheorem{theorem}{Theorem}[section]

\newtheorem{headthm}{Theorem}

\newaliascnt{headcor}{headthm}

\aliascntresetthe{headcor}

\newaliascnt{headconj}{headthm}

\aliascntresetthe{headconj}

\newaliascnt{corollary}{theorem}
\newtheorem{corollary}[corollary]{Corollary}
\aliascntresetthe{corollary}

\newaliascnt{claim}{theorem}

\aliascntresetthe{claim}

\newaliascnt{lemma}{theorem}
\newtheorem{lemma}[lemma]{Lemma}
\aliascntresetthe{lemma}

\newaliascnt{conjecture}{theorem}

\aliascntresetthe{conjecture}

\newaliascnt{proposition}{theorem}
\newtheorem{proposition}[proposition]{Proposition}
\aliascntresetthe{proposition}

\theoremstyle{definition}
\newaliascnt{definition}{theorem}
\newtheorem{definition}[definition]{Definition}
\aliascntresetthe{definition}

\newaliascnt{notation}{theorem}

\aliascntresetthe{notation}

\newaliascnt{example}{theorem}

\aliascntresetthe{example}

\newaliascnt{examples}{theorem}

\aliascntresetthe{examples}

\newaliascnt{remark}{theorem}
\newtheorem{remark}[remark]{Remark}
\aliascntresetthe{remark}

\newaliascnt{question}{theorem}

\aliascntresetthe{question}

\newaliascnt{questions}{theorem}

\aliascntresetthe{questions}

\newaliascnt{problem}{theorem}

\aliascntresetthe{problem}

\newaliascnt{construction}{theorem}

\aliascntresetthe{construction}

\newaliascnt{setup}{theorem}
\newtheorem{setup}[setup]{Setup}
\aliascntresetthe{setup}

\newaliascnt{algorithm}{theorem}

\aliascntresetthe{algorithm}

\newaliascnt{observation}{theorem}

\aliascntresetthe{observation}

\newaliascnt{defprop}{theorem}

\aliascntresetthe{defprop}

\def\equationautorefname~#1\null{(#1)\null}
\def\sectionautorefname~#1\null{Section #1\null}
\def\subsectionautorefname~#1\null{\S #1\null}

\begin{document}

\title{Fiber-full modules and a local freeness criterion for local cohomology modules}

\author{Yairon Cid-Ruiz}
\address{Department of Mathematics, KU Leuven, Celestijnenlaan 200B, 3001 Leuven, Belgium}
\email{yairon.cidruiz@kuleuven.be}

\begin{abstract}
	Let $R$ be a finitely generated positively graded algebra over a Noetherian local ring $B$, and $\mm = [R]_+$ be the graded irrelevant ideal of $R$.
	We provide a local criterion characterizing the $B$-freeness of all the local cohomology modules $\HL^i(M)$ of a finitely generated graded $R$-module $M$.
	We show that fiber-full modules are exactly the ones that satisfy this criterion.	
	When we change $B$ by an arbitrary Noetherian ring $A$, we study the fiber-full locus of a module in $\Spec(A)$: we show that the fiber-full locus is always an open subset of $\Spec(A)$ and that it is dense when $A$ is generically reduced.
\end{abstract}



\maketitle



\section{Introduction}

Let $(B, \bb)$ be a Noetherian local ring, $R$ be a finitely generated positively graded $B$-algebra, and $\mm = [R]_+$ be the graded irrelevant ideal of $R$.  
This paper is motivated by the following two lines of research:

\begin{enumerate}[(a)]
	
	\item \label{line_a}
	Recently there has been a lot of interest in the closely related notions of \emph{algebras having liftable local cohomology} introduced by Koll\'ar and Kov\'acs \cite{KOLLAR_KOVACS}, and \emph{cohomologically full rings} introduced by Dao, De Stefani and Ma \cite{COHOM_FULL_RINGS}.
	The study of these concepts has produced a number of important results, see, e.g., \cite{KOLLAR_KOVACS, COHOM_FULL_RINGS, CONCA_VARBARO}.
	Of great interest to us is the important work of Koll\'ar and Kov\'acs \cite{KOLLAR_KOVACS} on the flatness and base change of the cohomologies of relative dualizing complexes.
	
	For the case of modules, the correct extension of these notions seems to be the notion of \emph{fiber-full modules} that was recently coined by Varbaro in \cite{CONFERENCE_LEVICO} (also, see \cite{HONGMIAO}).
	Here we introduce and study a notion of fiber full-modules which is relative over the base ring $B$.
	We say that a finitely generated graded $R$-module $M$ is \emph{fiber-full over $B$} if $M$ is a free $B$-module and the natural map 
	$$
	\HL^i\big(M/\bb^qM\big) \;\rightarrow\; \HL^i\big(M/\bb M\big)
	$$ 
	is surjective for all $i \ge 0$ and $q \ge 1$ (see \autoref{nota_fiber_full}; cf. \cite[Definition 3.8]{CONFERENCE_LEVICO}).
	
	\smallskip 
	\item \label{line_b}
	An important question is to characterize when the local cohomology modules $\HL^i(M)$ are free $B$-modules for a finitely generated graded $R$-module $M$.
	This question has been addressed in various contexts by Hochster and Roberts  \cite[Theorem 3.4]{HOCHSTER_ROBERTS_PURITY}, by Koll\'ar \cite[Theorem 78]{kollar2014maps}, by Smith \cite{KSMITH}, and by Chardin and Simis in joint work \cite{GEN_FREENESS_LOC_COHOM} with the author of this paper.
\end{enumerate}

Our main result is the following theorem that gives a local criterion characterizing the $B$-freeness of all the local cohomology modules $\HL^i(M)$.  
This result shows that fiber-full modules are \emph{exactly} the ones that enjoy the property that $M$ is a free $B$-module and $\HL^i(M)$ is a free $B$-module for all $i \ge 0$.
So, in a sense, it paves a bridge between \hyperref[line_a]{(a)} and  \hyperref[line_b]{(b)}.

\begin{headthm}[\autoref{thm_freeness_criterion}]
	\label{thmA}
	Let $(B, \bb)$ be a Noetherian local ring, $R$ be a finitely generated positively graded $B$-algebra, and $\mm = [R]_+$.  
	Let $T=B[x_1,\ldots,x_r]$ be a positively graded polynomial ring over $B$ such that we have a homogeneous surjection $T \twoheadrightarrow R$.
	Let $M$ be a finitely generated graded $R$-module and suppose that $M$ is a free $B$-module.
	Then, the following six conditions are equivalent:
	\begin{enumerate}[\rm (1)]
		\item $\HL^i(M)$ is a free $B$-module for all $0 \le i \le r$.
		\item \label{part2} $\Ext_T^i(M, T)$ is a free $B$-module for all $0 \le i \le r$.
		\item $\HL^i(M/\bb^q M)$ is a free $B/\bb^q$-module for all $0 \le i \le r$ and $q \ge 1$.
		\item $\Ext_{T/\bb^qT}^i(M/\bb^q M,  T/\bb^qT)$ is a free $B/\bb^q$-module for all $0 \le i \le r$ and $q \ge 1$.
		\item The natural map 
		$
		\HL^i\big(M/\bb^qM\big) \rightarrow \HL^i\big(M/\bb M\big)
		$ 
		is surjective for all $0 \le i \le r$ and  $q \ge 1$.
		\item The natural map 
		$
		\Ext_{T/\bb^qT}^{i}(M/\bb M, \omega_{T/\bb^qT})
		\rightarrow 
		\Ext_{T/\bb^qT}^{i}(M/\bb^qM, \omega_{T/\bb^qT})
		$ 
		is injective for all $0 \le i \le r$ and $q \ge 1$, where  $\omega_{T/\bb^qT}$ denotes the graded canonical module of $T/\bb^qT$.
	\end{enumerate}
	Moreover, when any of the above equivalent conditions is satisfied, we have the following isomorphisms 
	\begin{enumerate}[\rm (i)]
		\item $\HL^i(M) \; \cong \; \big(\Ext_T^{r-i}(M, T(-\delta))\big)^{*_B}$ where $\delta = \deg(x_1) + \cdots + \deg(x_r)$,
		\item $\HL^i(M) \otimes_B C \; \xrightarrow{\cong} \; \HL^i(M \otimes_B 
		C)$, and
		\item \label{partiii}  $\Ext_T^i(M, T) \otimes_B C \;\xrightarrow{\cong}\; \Ext_{T \otimes_B C}^i(M\otimes_B C, T\otimes_B C)$
	\end{enumerate}	
	for all integers $i$ and any $B$-algebra $C$.
\end{headthm}

A couple of words regarding \autoref{thmA} are in place.
Because of the above theorem, fiber-full modules get six equivalent desirable definitions, and this makes them quite a versatile class of modules.
For fiber-full modules, we obtain a graded local duality theorem relative to the base ring $B$ (see \autoref{def_relative_Matlis}) and base change isomorphisms for all $\HL^i(M)$ and $\Ext_T^i(M, T)$.
Consider the morphism $f : X = \Spec(R) \rightarrow Y = \Spec(B)$ and the relative dualizing complex $\omega_{X/Y}^\bullet$, and suppose that $f$ is flat.
Since $f$ is embeddable into the smooth morphism $\mathbb{A}_B^r = \Spec(T) \rightarrow \Spec(B)$, it follows that the $-i$-th cohomology $h^{-i}(\omega_{X/Y}^\bullet)$ of $\omega_{X/Y}^\bullet$ coincides with the sheafification of $\Ext_T^{r-i}(R, T)$.
Therefore, in our current setting when we consider the special case $M = R$, parts \hyperref[part2]{(2)} and \hyperref[partiii]{(iii)} of \autoref{thmA} recover the result of Koll\'ar and Kov\'acs \cite{KOLLAR_KOVACS} on the flatness and base change of $h^{-i}(\omega_{X/Y}^\bullet)$.

\smallskip

Next, we concentrate on some applications of the theorem above. 

\smallskip

We now work over an arbitrary Noetherian ring, and we are interested in studying the locus of fiber-fullness of a module.
Let $A$ be a Noetherian ring, $R$ be a finitely generated positively graded $A$-algebra, and  $\mm = [R]_+$.  
For any $\pp \in \Spec(A)$, let $k(\pp) = A_\pp/\pp A_\pp$ be the residue field of $\pp$.
For a finitely generated graded $R$-module $M$, we define the \emph{fiber-full locus} of $M$ by
$$
\mathfrak{Fib}(M) \; = \; \big\{ \pp \in \Spec(A) \mid M \otimes_{A} A_\pp \text{ is fiber-full over } A_\pp \big\}.
$$
Interestingly, it turns out that the fiber-full locus is very well-behaved: $\mathfrak{Fib}(M)$ is always an open subset, and if $A$ is generically reduced then $\mathfrak{Fib}(M)$ is also dense.
The following theorem gives some results regarding the fiber-full locus of a module.

\begin{headthm}[\autoref{thm_locus}]	
	\label{thmB}
	Let $A$ be a Noetherian ring, $R$ be a finitely generated positively graded $A$-algebra, and  $\mm = [R]_+$.  
		Let $M$ be a finitely generated graded $R$-module.
	Then, the following statements hold:
	\begin{enumerate}[\rm (i)]
		\item $\mathfrak{Fib}(M)$ is an open subset of $\Spec(A)$.
		\item If $A$ is generically reduced, then there is an element $a \in A$ avoiding the minimal primes of $A$ such that $\HL^i(M) \otimes_{A} A_a$ is a projective $A_a$-module for all $i \ge 0$, and so $\mathfrak{Fib}(M)$ is a dense subset of $\Spec(A)$.
		\item For all $i \ge 0$ and $\nu \in \ZZ$ the function
		$$
		\Spec(A) \rightarrow \NN, \quad \pp \;\mapsto\; \dim_{k(\pp)}\left( \left[\HL^i(M \otimes_{A} k(\pp))\right]_\nu \right)
		$$		
		is locally constant on $\mathfrak{Fib}(M)$.
	\end{enumerate}
\end{headthm}

Finally, as a consequence of our methods, in \autoref{thm_grobner_deform} we give an alternative proof of an important result of Conca and Varbaro \cite{CONCA_VARBARO}, where they showed that for a given homogeneous ideal $I \subset S = \kk[x_1,\ldots,x_r]\; (\mm = (x_1,\ldots,x_r))$ the Hilbert functions of $\HL^i(S/I)$ and $\HL^i(S/\iniTerm_<(I))$ coincide provided the initial ideal $\iniTerm_<(I)$ is square-free.
This result settled a previous conjecture of Herzog.

\smallskip

The basic outline of this paper is as follows.
In \autoref{sect_criterion}, we prove our main result, that is, we give the  proof of \autoref{thmA}.
In \autoref{sect_fib_full}, we obtain \autoref{thmB}.
Finally, in \autoref{sect_grob_degen}, we study square-free Gr\"obner degenerations.

\section{A local freeness criterion for local cohomology  modules}
\label{sect_criterion}

In this section, we provide a local criterion for the freeness of local cohomology modules.
Throughout this section, the following setup is set in place.

\begin{setup}
	\label{setup_local_criterion}
	Let $(B, \bb)$ be a Noetherian local ring.
	Let $R$ be a finitely generated positively graded $B$-algebra.
	Let $T = B[x_1,\ldots,x_r]$ be a positively graded polynomial ring over $B$ (that is, $[T]_0 = B$ and $\deg(x_i) > 0$),  such that we have a homogeneous surjection $T \twoheadrightarrow R$.
	Let $\mm = (x_1, \dots, x_r) \subset T$ be the graded irrelevant ideal of $T$.
	By abuse of notation, interchangeably, $\mm$ will also denote the graded irrelevant ideal $\mm R = [R]_+ = \bigoplus_{\nu \ge 1} [R]_\nu$ of $R$.
	For each $q \ge 1$, let $B_q := B/\bb^q$,  $R_q := R \otimes_{B} B_q$ and  $T_q := T \otimes_{B} B_q\cong B_q[x_1,\ldots,x_r]$.
	Let $\delta :=  \deg(x_1) + \cdots + \deg(x_r) \in \NN$.
\end{setup}

\begin{remark}\label{rem_basic}
	We begin by recalling the following basic facts:
	\begin{enumerate}[(i)]
		\item For any $R$-module $M$, by the Independence Theorem (see \cite[Theorem 4.2.1]{Brodmann_Sharp_local_cohom}) we can compute  $\HL^i(M)$ by considering $M$ as a $T$-module.
		\item $\HL^{r}(T) \cong \frac{1}{x_1\cdots x_r}B[x_1^{-1},\ldots,x_r^{-1}]$ and $\HL^i(T) = 0$ for all $i \neq r$.
		\item Since $B$ is a Noetherian local ring, there is no distinction between the notions of $B$-flat, $B$-projective and $B$-free for a finitely generated $B$-module (see, e.g., \cite[\S 7]{MATSUMURA}).
		It then follows that the properties of $B$-flat, $B$-projective and $B$-free are equivalent for a  finitely generated graded $R$-module. 
		\item $\HL^i(M)=0$ for all $i \ge r+1$ and any $R$-module $M$.
		\item If $M$ is a finitely generated graded $R$-module, then for each $\nu \in \ZZ$ the graded component ${[\HL^i(M)]}_\nu$ is a finitely generated $B$-module (see, e.g., \cite[Theorem III.5.2]{HARTSHORNE}, \cite[Theorem 2.1]{CHARDIN_POWERS_IDEALS}).
		So the conditions $B$-flat, $B$-projective and $B$-free are equivalent for  $\HL^i(M)$.
	\end{enumerate}
\end{remark}

Our main object of study is the following interesting class of modules (cf. \cite[Definition 3.8]{CONFERENCE_LEVICO}). 
We introduce a notion which is relative over the base ring $B$.

\begin{definition}
	\label{nota_fiber_full}
	 A finitely generated graded $R$-module $M$ is \emph{fiber-full over $B$} if $M$ is a free $B$-module and the natural map 
		$
		\HL^i(M / \bb^qM) \rightarrow \HL^i(M / \bb M)
		$
		is surjective for all $i \ge 0$ and $q \ge 1$.
\end{definition}

An equivalent definition for the notion of fiber-full is given by the following lemma.

\begin{lemma}
	\label{lem_fiber_full_def_Ext}
	Let $M$ be a finitely generated graded $R$-module.
	Then, the natural map 
	$
	\HL^i(M / \bb^qM) \rightarrow \HL^i(M / \bb M)
	$
	is surjective if and only if the natural map 
	$
	\Ext_{T_q}^{r-i}(M/\bb M, \omega_{T_q})
	\rightarrow 
	\Ext_{T_q}^{r-i}(M/\bb^qM, \omega_{T_q})
	$ 
	is injective, where $\omega_{T_q}$ denotes the graded canonical module of $T_q$.
\end{lemma}
\begin{proof}
	Note that $T_q$ is a Cohen-Macaulay ${}^{*}$complete ${}^{*}$local ring with ${}^{*}$maximal ideal $(\bb + \mm) T_q$ of ${}^{*}$dimension $r$ (under the notations of \cite{BRUNS_HERZOG}).
	Thus, the graded local duality theorem \cite[Theorem  3.6.19]{BRUNS_HERZOG} implies that the natural map $\HH_{\bb+\mm}^i(M / \bb^qM) \rightarrow \HH_{\bb+\mm}^i(M / \bb M)$ is surjective if and only if the natural map $\Ext_{T_q}^{r-i}(M/\bb M, \omega_{T_q})
	\rightarrow 
	\Ext_{T_q}^{r-i}(M/\bb^qM, \omega_{T_q})$ is injective.
	Since $M / \bb^qM$ is $\bb$-torsion, the spectral sequence $E_2^{i,j} = \HL^i\big( \HH_\bb^j(M / \bb^qM) \big) \Rightarrow \HH_{\bb+\mm}^{i+j}(M / \bb^qM)$ yields the isomorphism $\HL^i(M / \bb^qM) \cong \HH_{\bb+\mm}^i(M / \bb^qM)$.
	So, the result is clear.
\end{proof}

An explicit computation of the graded canonical module $\omega_{T_q}$ is given in the remark below.

\begin{remark}
	\label{rem_canon_mod}
	As in the proof of \autoref{lem_fiber_full_def_Ext}, we have $\HH_{\bb+\mm}^r(T_q) \cong \HL^r(T_q) \cong \frac{1}{x_1\cdots x_r}B_q[x_1^{-1},\ldots,x_r^{-1}]$.
	Consequently, \cite[Theorem  3.6.19]{BRUNS_HERZOG} gives the following isomorphism
	$$
	\omega_{T_q} \;\cong\; \bigoplus_{\nu \in \ZZ}  \Hom_{B_q}\left(\left[\frac{1}{x_1\cdots x_r}B_q[x_1^{-1},\ldots,x_r^{-1}]\right]_{-\nu}, \omega_{B_q}\right) \;\cong \; T_q(-\delta) \otimes_{B_{q}} \omega_{B_q}
	$$
	where $\omega_{B_q}$ denotes the canonical module of the Artinian local ring $B_q$.
\end{remark}

For completeness, we show below that \autoref{nota_fiber_full} agrees with \cite[Definition 3.8]{CONFERENCE_LEVICO} under a general common setting.

\begin{remark}
	Let $C = \kk[t]$ and  $P = C[x_1,\ldots,x_r]$ be a positively graded polynomial ring over $C$.
	In \cite[Definition 3.8]{CONFERENCE_LEVICO} a finitely generated graded $P$-module $M$ is called \emph{fiber-full} if $M$ is $C$-flat and the natural map $\Ext_P^i(M/tM,P) \rightarrow \Ext_P^i(M/t^qM,P)$ is injective for all $i \ge 0, q \ge 1$.
	By a theorem of Rees (see, e.g., \cite[Lemma 3.1.16]{BRUNS_HERZOG}) this is equivalent to saying that $\Ext_{P/t^qP}^{i-1}(M/tM,P/t^qP) \rightarrow \Ext_{P/t^qP}^{i-1}(M/t^qM,P/t^qP)$ is injective for all $i \ge 0, q \ge 1$.
	Thus the sought equivalence between \autoref{nota_fiber_full} and \cite[Definition 3.8]{CONFERENCE_LEVICO} is given by \autoref{lem_fiber_full_def_Ext} and \autoref{rem_canon_mod}.
\end{remark}

\begin{definition}
	\label{def_relative_Matlis}
	For a graded $T$-module  $M$, we denote the \emph{$B$-relative graded Matlis dual} by
	$$
	\left(M\right)^{*_B}={}^*\Hom_B(M, B): = \bigoplus_{\nu \in \ZZ} \Hom_B\left({\left[M\right]}_{-\nu}, B\right).
	$$ 
\end{definition}

Note that $\left(M\right)^{*_B}$ has a natural structure of graded $T$-module.
From the canonical perfect pairing of free $B$-modules in ``top'' local cohomology
$$
{\left[T\right]}_\nu \otimes_{B} {\left[\HL^{r}(T)\right]}_{-\delta-\nu} \rightarrow {\left[\HL^{r}(T)\right]}_{-\delta} \cong B
$$
we obtain a canonical graded $R$-isomorphism 
$
\HL^{r}(T) \cong {\left(T(-\delta)\right)}^{*_B} = \grHom_B\left(T(-\delta), B\right).	
$
Then, for a graded complex $F_\bullet$ of finitely generated free $T$-modules, we obtain the isomorphism of complexes 
\begin{equation}
	\label{eq_isom_complexes}
	\HL^{r}(F_\bullet) \;\cong\; {\Big(\Hom_T(F_\bullet
		, T(-\delta))\Big)}^{*_B}.
\end{equation}
The lemma below is an important tool for us.
\begin{lemma}
	\label{lem_compute_loc_cohom}
	Let $M$ be a finitely generated graded $R$-module.
	Let $F_\bullet$ be a graded free $T$-resolution of $M$ by modules of finite rank.
	If $M$ is a free $B$-module, then  
	$
	\HL^i(M \otimes_{B} C) \cong \HH_{r-i}\big(\HL^{r}(F_\bullet) \otimes_{B} C\big)
	$
	for any $B$-algebra $C$ and all integers $i$.
\end{lemma}
\begin{proof}
	See \cite[Lemma 3.4]{GEN_FREENESS_LOC_COHOM}.
\end{proof}

The next  lemma provides a known base change isomorphism for local cohomology.

\begin{lemma}
	\label{lem_base_ch_loc}
	Let $M$ be a finitely generated graded $R$-module.
	Suppose that $M$ is a free $B$-module and $\HL^i(M)$ is a free $B$-module for all $i \ge 0$.
	Then 
	$$
	\HL^i(M)  \otimes_{B} C  \;\xrightarrow{\cong}\; \HL^i(M \otimes_{B} C)
	$$ 
	for all $i \ge 0$ and any $B$-algebra $C$.
\end{lemma}
\begin{proof}
	Let $F_\bullet$ be a graded free $T$-resolution of $M$ by modules of finite rank, and set $G_\bullet := \HL^{r}(F_\bullet)$.
	By \autoref{lem_compute_loc_cohom}, $\HH_{i}(G_\bullet)$ is a free $B$-module for all $i \ge 0$.
	Then, it follows that $\HH_i(G_\bullet) \otimes_{B} C \xrightarrow{\cong} \HH_i(G_\bullet \otimes_{B} C)$ for all $i \ge 0$ (see, e.g., \cite[Lemma 2.8]{GEN_FREENESS_LOC_COHOM}).
	By using \autoref{lem_compute_loc_cohom} again we obtain 
	$$
	\HL^i(M)\otimes_{B} C \;\cong\; \HH_{r-i}(G_\bullet) \otimes_{B} C \;\xrightarrow{\cong}\; \HH_{r-i}(G_\bullet \otimes_{B} C) \;\cong\;  \HL^i(M \otimes_{B} C),
	$$
	and so the result follows.
\end{proof}

By using the property of exchange for local Ext's (see \cite[Theorem 1.9]{ALTMAN_KLEIMAN_COMPACT_PICARD}, \cite[Theorem A.5]{LONSTED_KLEIMAN}), we obtain the following lemma.

\begin{lemma}
	\label{lem_Ext_base_change}
	Let $M$ be a finitely generated graded $R$-module and suppose that $M$ is a free $B$-module.
	Let $F_\bullet : \: \cdots \rightarrow F_i \rightarrow \cdots \rightarrow  F_1 \rightarrow F_0$ be a graded free $T$-resolution of $M$ by modules of finite rank.
	Let 
	$$
	D_M^i \;:=\; \Coker\big(\Hom_T(F_{i-1}, T) \rightarrow \Hom_T(F_{i}, T)\big)
	$$
	for each $i \ge 0$.
	Then, the following statements hold:
	\begin{enumerate}[\rm (i)]
		\item $\Ext_T^i(M, T) = 0$ for all $i \ge r+1$.
		\item $D_M^i$ is a free $B$-module for all $i \ge r+1$.
		\item If $\Ext_T^i(M, T)$ is a free $B$-module for all $0 \le i \le r$, then $$
		\Ext_T^i(M, T) \otimes_B C \;\xrightarrow{\cong}\; \Ext_{T \otimes_B C}^i(M\otimes_B C, T\otimes_B C)
		$$
		for all $i \ge 0$ any $B$-algebra $C$.
	\end{enumerate}
\end{lemma}
\begin{proof}
	(i)
	Since $M$ is a free $B$-module, $F_\bullet \otimes_B B/\bb$ is a free $T_1$-resolution of $M/\bb M$.
	By using Hilbert's syzygy theorem on the polynomial ring $T_1 \cong B/\bb[x_1,\ldots,x_r]$, we get that $\Ext_{T_1}^i(M/\bb M, T_1) = 0$ for all $i \ge r+1$.
	In particular, the natural map 
	$$
	\Ext_T^{i}(M, T) \otimes_B B/\bb \rightarrow \Ext_{T_1}^i(M/\bb M, T_1) = 0
	$$
	is surjective for all $i \ge r+1$, however, \cite[Theorem 1.9]{ALTMAN_KLEIMAN_COMPACT_PICARD} or \cite[Theorem A.5]{LONSTED_KLEIMAN}  imply that the map above is actually bijective.
	As $\Ext_T^i(M, T)$ is a finitely generated graded $T$-module, each graded component $\left[\Ext_T^i(M, T)\right]_\nu$ is a finitely generated $B$-module.
	Note that the condition $\Ext_T^i(M, T) \otimes_B B/\bb = 0$ implies that 
	$
	\left[\Ext_T^i(M, T)\right]_\nu \otimes_B B/\bb = 0$
	 for all $\nu \in \ZZ$.
	Then, for all $i \ge r+1$, Nakayama’s lemma yields that $\left[\Ext_T^i(M, T)\right]_\nu = 0$ for all $\nu \in \ZZ$, and so $\Ext_T^i(M, T) = 0$, as required.
	
	(ii) 
	Since $M$ is $B$-free, by considering the complex $\Hom_T(F_\bullet, T) \otimes_{B} C$  with $C$ any $B$-algebra,  we obtain the four-term exact sequence 
	\begin{equation}
		\label{eq_four_term_exact}
		0 \rightarrow \Ext_{T \otimes_B C}^i(M\otimes_B C, T\otimes_B C) \rightarrow D_M^i \otimes_B C \rightarrow \Hom_T(F_{i+1}, T) \otimes_{B} C \rightarrow  D_M^{i+1} \otimes_{B} C \rightarrow 0
	\end{equation}
	for all $i \ge 0$.
	For all $i \ge r+1$, due to the computations of part (i), when we substitute $C = B$ and $C = B/\bb$ in \autoref{eq_four_term_exact}, we get the short exact sequences 
	\begin{equation}
		\label{eq_exact_seq_D_M}
		0 \rightarrow D_M^i \rightarrow \Hom_T(F_{i+1}, T)  \rightarrow  D_M^{i+1} \rightarrow 0
	\end{equation}	
	and 
	\begin{equation}
			\label{eq_tensor_exact_seq_D_M}
			0 \rightarrow D_M^i \otimes_{B} B/\bb \rightarrow \Hom_T(F_{i+1}, T) \otimes_{B} B/\bb \rightarrow  D_M^{i+1} \otimes_{B} B/\bb \rightarrow 0,
	\end{equation}
	respectively.
	For all $i \ge r+1$, from \autoref{eq_tensor_exact_seq_D_M} and the long exact sequence in Tor's induced by tensoring \autoref{eq_exact_seq_D_M} with $B/\bb$, we obtain that $\Tor_1^B(D_M^{i+1}, B/\bb) = 0$, and so by the local flatness criterion (see \cite[Theorem 22.3]{MATSUMURA}) if follows that $D_M^{i+1}$ and $D_M^i$ are free $B$-modules.
	
	(iii)
	Let $\gamma^i$ be the natural map $\gamma^i : \Ext_T^i(M, T) \otimes_{B} B/\bb \rightarrow \Ext_{T_1}^i(M/\bb M, T_1)$.
	By \cite[Theorem 1.9]{ALTMAN_KLEIMAN_COMPACT_PICARD} or \cite[Theorem A.5]{LONSTED_KLEIMAN}, we have that under the assumption that $\gamma^i$ is surjective, then $\gamma^{i-1}$ is surjective if and only if $\Ext_T^i(M, T)$ is a free $B$-module.
	Since $\gamma^{r+1}$ is surjective and $\Ext_T^i(M, T)$ is a free $B$-module for all $0 \le i \le r$, by descending induction on $i$ we get that $\gamma^i$ is surjective for all $i \ge 0$.
	Finally, by using \cite[Theorem 1.9]{ALTMAN_KLEIMAN_COMPACT_PICARD} or \cite[Theorem A.5]{LONSTED_KLEIMAN} the result follows.
\end{proof}

The next proposition yields some version of graded local duality relative to the base ring $B$.

\begin{proposition}
	\label{prop_equiv_freeness}
	Let $M$ be a finitely generated graded $R$-module and suppose that $M$ is a free $B$-module.
	Then, the following two conditions are equivalent:
	\begin{enumerate}[\rm (1)]
		\item $\HL^i(M)$ is a free $B$-module for all $0 \le i \le r$.
		\item $\Ext_T^i(M, T)$ is a free $B$-module for all $0 \le i \le r$.
	\end{enumerate}
	Moreover, when any of the above equivalent conditions is satisfied, we have that 
	$$
	\HL^i(M) \cong \left(\Ext_T^{r-i}(M, T(-\delta))\right)^{*_B}
	$$ 
	for all integers $i$.
\end{proposition}
\begin{proof}
	Let $F_\bullet: \cdots \rightarrow F_i \rightarrow \cdots \rightarrow  F_1 \rightarrow F_0$ be a graded free $T$-resolution of $M$ by modules of finite rank.
	Let $I^\bullet:  I^0 \rightarrow I^1 \rightarrow \cdots \rightarrow I^i \rightarrow \cdots$ be an injective $B$-resolution of $B$.
	
	 $(1) \Rightarrow (2)$.	
	Suppose that $\HL^i(M)$ is a free $B$-module for all $i \ge 0$ (see \autoref{rem_basic}(iv)).	
	To simplify the notation, let $G_\bullet := \HL^{r}(F_\bullet)$.
	By \autoref{lem_compute_loc_cohom}, $\HH_{i}(G_\bullet)$ is a free $B$-module for all $i \ge 0$.
	From \autoref{eq_isom_complexes}, we obtain the isomorphism of complexes $\left(G_\bullet\right)^{*_B} \cong \Hom_T(F_\bullet
	, T(-\delta))$.
	Thus, to finish the proof it suffices to show that $\HH^{i}\big(\left(G_\bullet\right)^{*_B}\big) \cong \big(\HH_i(G_\bullet)\big)^{*_B}$ because $\Ext_T^i(M, T(-\delta)) \cong \HH^i\big(\Hom_T(F_\bullet
	, T(-\delta))\big)$.
		
	Let $K^{\bullet, \bullet}$ be the first quadrant double complex given by 
	$
	K^{p,q} := \bigoplus_{\nu \in \ZZ} \Hom_B\left([G_p]_{-\nu}, I^q\right).
	$
	Note that the corresponding spectral sequences have respective second terms
	$$
	{}^\text{I}E_2^{p,q} = \bigoplus_{\nu \in \ZZ} \HH^{p}\big(\Ext^q_B ([G_\bullet]_{-\mu}, B)\big) \quad\text{ and }\quad
	{}^\text{II}E_2^{p,q} = \bigoplus_{\nu \in \ZZ} \Ext^p_B \big([\HH_q(G_\bullet)]_{-\nu},B\big)  .
	$$  
	Finally, since $G_i$ and $\HH_{i}(G_\bullet)$ are free $B$-modules for all $i \ge 0$, we obtain the isomorphisms 
	$$
	\Ext_T^i(M, T(-\delta)) \; \cong \; \HH^{i}\big(\left(G_\bullet\right)^{*_B}\big) \;\cong\; \big(\HH_i(G_\bullet)\big)^{*_B} \;\cong \; \big( \HL^{r-i}(M) \big)^{*_B}
	$$  
	for all $i \ge 0$.

	$(2) \Rightarrow (1)$.	
	Suppose that $\Ext_T^i(M, T)$ is a free $B$-module for all $0 \le i \le r$.	
	Let $F_\bullet^{\le r+1}$ be the complex given as the truncation
	$$
	F_\bullet^{\le r+1} \; : \; 0 \rightarrow F_{r+1} \rightarrow F_r \rightarrow \cdots \rightarrow F_1 \rightarrow F_0, 
	$$
	and $P^\bullet := \Hom_T(F_\bullet^{\le r+1}, T(-\delta))$.
	By hypothesis $\HH^i(P^\bullet) \cong \Ext_T^i(M, T)(-\delta)$ is a free $B$-module for all $0 \le i \le r$.
	From \autoref{lem_Ext_base_change}, we also have that $\HH^{r+1}(P^\bullet) \cong D_M^{r+1}(-\delta)$  is a free $B$-module.
	
	Let $L^{\bullet, \bullet}$ be the second quadrant double complex given by 
	$
	L^{-p,q} := \bigoplus_{\nu \in \ZZ} \Hom_B\left([P^p]_{-\nu}, I^q\right).
	$
	Now the corresponding spectral sequences have respective second terms
	$$
	{}^\text{I}E_2^{-p,q} = \bigoplus_{\nu \in \ZZ} \HH^{-p}\big(\Ext^q_B ([P^\bullet]_{-\mu}, B)\big) \quad\text{ and }\quad
	{}^\text{II}E_2^{p,-q} = \bigoplus_{\nu \in \ZZ} \Ext^p_B \big([\HH^q(P^\bullet)]_{-\nu},B\big)  .
	$$  
	Finally, since $P^i$ and $\HH^{i}(P^\bullet)$ are free $B$-modules for all $0 \le i \le r+1$, by using \autoref{lem_compute_loc_cohom} and \autoref{eq_isom_complexes} we  obtain the following isomorphisms 
	$$
	\HL^{r-i}(M)	\;\cong\;  \HH_i\big((P^\bullet)^{*_B}\big)  \;\cong\;  \big(\HH^i(P^\bullet)\big)^{*_B} \;\cong\; \big(\Ext_T^i(M, T(-\delta))\big)^{*_B}
	$$
	for all $0 \le i \le r$.
	This concludes the proof of the proposition.
\end{proof}

We now concentrate on giving a ``Theorem on Formal Functions'' for the local cohomology modules.
Our proof of this result is a consequence of the well-known Theorem on Formal Functions for sheaf cohomology (see \cite[Theorem III.11.1]{HARTSHORNE}, \cite[\href{https://stacks.math.columbia.edu/tag/02OC}{Tag 02OC}]{stacks-project}).
Let $M$ be a finitely generated graded $R$-module. 
Note that for every $q \ge 1$ and $\nu \in \ZZ$, we have a natural map ${[\HL^i(M)]}_\nu \otimes_{B} B_q \rightarrow {[\HL^i(M/\bb^qM)]}_\nu$.
Then, we obtain an induced natural map 
$$
	{[\HL^i(M)]}_\nu\,{\widehat{}} \;\,\rightarrow\;\, \underset{\leftarrow}{\lim}\, {[\HL^i(M/\bb^qM)]}_\nu,
$$
where ${[\HL^i(M)]}_\nu\,{\widehat{}}$ \, denotes the completion of the finitely generated $B$-module ${[\HL^i(M)]}_\nu$ with respect to the maximal ideal $\bb \subset B$.

\begin{proposition}
	\label{prop_formal_funct}
	Let $M$ be a finitely generated graded $R$-module.
	Then, the natural map 
	$$
	{[\HL^i(M)]}_\nu\,{\widehat{}} \;\; \xrightarrow{\cong} \;\; \underset{\leftarrow}{\lim}\, {[\HL^i(M/\bb^qM)]}_\nu,
	$$
	is an isomorphism for all $i \ge 0$ and $\nu \in \ZZ$.
\end{proposition}
\begin{proof}
	If we set $T' = T[x_{r+1}]$ and $M' = M \otimes_{B} B[x_{r+1}]$, then for all $i \ge 0$ we get the natural isomorphism $\HH_{(\mm + x_{r+1}T')}^i\big(M'/x_{r+1}M'\big) \cong \HL^i(M)$.
	Thus, without any loss of generality, we assume that $r \ge 2$. 
	
	Let $X = \Proj(T)$ (here we may need $r \ge 2$) and $\FF = \widetilde{M}$ be the corresponding coherent sheaf.
	For each $q \ge 1$, let $M_q := M / \bb^q M$, $X_q  := X \times_{\Spec(B)} \Spec(B_q)$ and $\FF_q := \FF \otimes_B B_q$.
	For $i \ge 2$, \cite[Theorem III.11.1]{HARTSHORNE} or \cite[\href{https://stacks.math.columbia.edu/tag/02OC}{Tag 02OC}]{stacks-project} yield the following natural isomorphism 
	$$
	{[\HL^i(M)]}_\nu\;{\widehat{}} \;\, \cong\;\, \HH^{i-1}(X, \FF(\nu)) \;\, {\widehat{}} \;\, \xrightarrow{\cong\;} \;\, \underset{\leftarrow}{\lim}\, \HH^{i-1}(X_q, \FF_q(\nu)) \;\, \cong \;\, \underset{\leftarrow}{\lim}\, {[\HL^i(M_q)]}_\nu.
	$$
	(For the relations between sheaf and local cohomologies see, e.g., \cite[Theorem A4.1]{EISEN_COMM}, \cite[Corollary 1.5]{HYRY_MULTIGRAD}).
	
	For $i \le 1$, we have the following commutative diagram with natural maps
	\begin{equation*}		
		\begin{tikzpicture}[baseline=(current  bounding  box.center)]
			\matrix (m) [matrix of math nodes,row sep=3em,column sep=1.5em,minimum width=2em, text height=1.5ex, text depth=0.25ex]
			{
				0 &  \text{${[\HL^0(M)]}_\nu$} \;{\widehat{}} & \text{${[M]}_\nu\;{\widehat{}}$}  & \HH^0(X, \FF(\nu))\;{\widehat{}} & \text{${[\HL^1(M)]}_\nu\;{\widehat{}}$} & 0\\
				0 & \text{$\underset{\leftarrow}{\lim}\, {[\HL^0(M_q)]}_\nu$} & \text{$\underset{\leftarrow}{\lim}\, {[M_q]}_\nu$}  &  \underset{\leftarrow}{\lim}\, \HH^0(X_q, \FF_q(\nu))  & \underset{\leftarrow}{\lim}\, {[\HL^1(M_q)]}_\nu & 0.\\
			};
			\path[-stealth]
			(m-1-1) edge (m-1-2)
			(m-1-2) edge (m-1-3)
			(m-1-3) edge (m-1-4)
			(m-1-4) edge (m-1-5)
			(m-1-5) edge (m-1-6)
			(m-2-1) edge (m-2-2)
			(m-2-2) edge (m-2-3)
			(m-2-3) edge (m-2-4)
			(m-2-4) edge (m-2-5)
			(m-2-5) edge (m-2-6)
			(m-1-2) edge node [left]  {$\alpha_1$} (m-2-2)		
			(m-1-3) edge node [left]  {$\alpha_2$} (m-2-3)
			(m-1-4) edge node [left]  {$\alpha_3$} (m-2-4)
			(m-1-5) edge node [left]  {$\alpha_4$} (m-2-5)
			;
		\end{tikzpicture}	
	\end{equation*}
	The first row is exact due to the exactness of completion.
	Note that ${[\HL^i(M_q)]}_\nu$ and $\HH^i(X_q, \FF_q(\nu))$ are $B_q$-modules of finite length for all $i \ge 0, \nu \in \ZZ$. 
	Thus, by using \cite[Proposition II.9.1]{HARTSHORNE} and \cite[Example II.9.1.2]{HARTSHORNE} (i.e., the Mittag Leffler condition is satisfied), we obtain that the second row is also exact.
	The map $\alpha_2$ is an isomorphism by the definition of completion and the map $\alpha_3$ is an isomorphism by \cite[Theorem III.11.1]{HARTSHORNE} or \cite[\href{https://stacks.math.columbia.edu/tag/02OC}{Tag 02OC}]{stacks-project}.
	It then follows that $\alpha_1$ and $\alpha_4$ are also isomorphisms.
	
	Therefore, the statement of the proposition holds for all $i \ge 0$.
\end{proof}

Next, we have a basic lemma regarding the behavior of fiber-full modules (see \autoref{nota_fiber_full}).
The following result is an extension of \cite[Proposition 4.5]{KOLLAR_KOVACS} and \cite[Proposition 2.2]{CONCA_VARBARO} to the case of modules.

\begin{lemma}
	\label{lem_filtrations}
	Let $C$ be an Artinian local ring which is a quotient of $B$. 
	Let $M$ be a finitely generated graded $R$-module and suppose that $M$ is fiber-full over $B$.
	\begin{enumerate}[\rm (i)]
		\item 
		Let $N$ be a $C$-module and $\phi : N \twoheadrightarrow B/\bb$ be a surjective homomorphism.
		 Then, the natural map 
			$$
			\HL^i(M \otimes_B \phi):\; \HL^i(M \otimes_B N) \rightarrow \HL^i(M/\bb M)
			$$ 
			is surjective for all $i \ge 0$.
		\item 
		Let $C = I_0 \supset I_1 \supset \cdots \supset I_{\ell-1} \supset I_{\ell} =0$ be a filtration with $I_{j}/I_{j+1} \cong B/\bb$  for all $0 \le j \le \ell-1$.
		 Then, we have the following short exact sequences
			$$
			0 \rightarrow \HL^ i(M \otimes_B I_{j+1}) \rightarrow  \HL^ i(M \otimes_B I_j)  \rightarrow \HL^ i(M \otimes_B I_j/I_{j+1}) \rightarrow 0
			$$
			for all $i\ge 0$ and $0 \le j \le \ell-1$.
	\end{enumerate}
\end{lemma}
\begin{proof}
	(i)
	Let $w \in N$ such that $\phi(w) = 1 \in B/\bb$.
	Take $q \ge 1$ such that  we have a surjection $B_q \twoheadrightarrow C$ and so $N$ becomes a $B_q$-module.
	Let $\eta : B_q \rightarrow N$ be the map defined by $1 \in B_q \mapsto w \in N$.
	Therefore, the natural map $\HL^i(M / \bb^qM) \rightarrow \HL^i(M/\bb M)$ coincides with the composition of maps  $\HL^i(M \otimes_B \phi) \circ \HL^i(M \otimes_B \eta)$, and so the result follows from the definition of fiber-full.

	(ii)  
	Fix $0 \le j \le \ell-1$.
	Since $M$ is a free $B$-module, we have the short exact sequence 
	$$
	0 \rightarrow M \otimes_B I_{j+1} \rightarrow  M \otimes_B I_j \rightarrow M \otimes_B I_j/I_{j+1} \rightarrow 0,
	$$ 
	and so we get the long exact sequence in local cohomology 
	$$
	\HL^ {i-1}(M \otimes_B I_j) \xrightarrow{\beta_{i-1}} \HL^ {i-1}(M \otimes_B I_j/I_{j+1}) \rightarrow \HL^ i(M \otimes_B I_{j+1}) \rightarrow  \HL^ i(M \otimes_B I_j)  \xrightarrow{\beta_i} \HL^ i(M \otimes_B I_j/I_{j+1}). 
	$$
	By part (i) the map $\beta_i$ is surjective for all $i \ge 0$, and this implies the result.
\end{proof}

The following proposition shows that fiber-fullness implies the freeness of certain local cohomology  modules. 
This result is inspired by the techniques  used in \cite[\S 4]{KOLLAR_KOVACS}.

\begin{proposition}
	\label{prop_free_impl_fiber_full}
	Let $M$ be a finitely generated graded $R$-module and suppose that $M$ is fiber-full over $B$.
	Then $\HL^i(M/\bb^qM)$ is a free $B_q$-module for all $i \ge 0$ and $q \ge 1$.
\end{proposition}
\begin{proof}
	Let $C$ be an Artinian local ring which is a quotient of $B$.
	Let $C = I_0 \supset I_1 \supset \cdots \supset I_{\ell-1} \supset I_{\ell} =0$ be a filtration with $I_{j}/I_{j+1} \cong B/\bb$ for all $0 \le j \le \ell-1$.
	For ease of notation, let $\overline{M} := M \otimes_B C$.
	As $M$ is a free $B$-module, we have $M \otimes_B I_{j}  \cong I_j\,\overline{M}$.	
	Note that $I_1 = \bb C$, $I_{\ell-1} \cong B/\bb$ and $I_{\ell-1} = (t) \subset C$ is a principal ideal generated by some element $0 \neq t \in C$.	
	
	 We shall prove that $\HL^i(M \otimes_B C)$ is a free $C$-module for all $i \ge 0$.
	We proceed by induction on the length $\ell = \text{length}(C)$ of $C$.
	If $\ell = 1$, then $C = B/\bb$ and the result is clear.
	Thus, we assume that $\ell \ge 2$ and that $\HL^i(M \otimes_B C')$ is a free $C'$-module for any quotient $C'$ of $B$ with $\text{length}(C') < \ell$.

	From \autoref{lem_filtrations}(ii), by successively composing the injections $\HL^ i(M \otimes_B I_{j+1}) \hookrightarrow  \HL^ i(M \otimes_B I_j)$, we obtain the injection $\HL^ i(M \otimes_B I_{j}) \hookrightarrow  \HL^ i(M \otimes_B C)$.	
	Then, $0 \rightarrow I_j\,\overline{M} \rightarrow \overline{M} \rightarrow \overline{M}/I_j\,\overline{M} \rightarrow 0$  yields a long exact sequence in local cohomology that splits into the following short exact sequences
	\begin{equation}
		\label{eq_ex_seq_HL_tM}
		0 \rightarrow \HL^i(I_j\,\overline{M}) \rightarrow \HL^i(\overline{M}) \rightarrow \HL^i(\overline{M}/I_j\,\overline{M}) \rightarrow 0,
	\end{equation}
	because $\HL^i(I_j\,\overline{M}) \cong \HL^ i(M \otimes_B I_{j}) \hookrightarrow  \HL^ i(M \otimes_B C) \cong \HL^i(\overline{M})$ is injective for all $i \ge 0$.
	Hence, multiplication by $t$ induces the following commutative diagram 
	\begin{equation*}		
		\begin{tikzpicture}[baseline=(current  bounding  box.center)]
			\matrix (m) [matrix of math nodes,row sep=2.3em,column sep=1.5em,minimum width=2em, text height=1.5ex, text depth=0.25ex]
			{
				\HL^i(\overline{M}) &  & \HL^i(\overline{M}) \\
				 & \HL^i(t\,\overline{M}) &\\
			};
			\path[-stealth]
			(m-1-1) edge node [above]  {$\cdot t$} (m-1-3)
			;
			\path [draw,->>] (m-1-1) -- (m-2-2);
			\draw[right hook->] (m-2-2)--(m-1-3);		
		\end{tikzpicture}	
	\end{equation*}
	(the surjectivity of $\HL^i(\overline{M}) \twoheadrightarrow \HL^i(t\, \overline{M})$ follows from \autoref{lem_filtrations}(i), and the injectivity of $\HL^i(t\overline{M}) \hookrightarrow \HL^i(\overline{M})$ is given by \autoref{eq_ex_seq_HL_tM} with $j = \ell -1$). 
	So we obtain that $t \, \HL^i(\overline{M})$ and $\HL^i(t\, \overline{M})$ coincide as submodules of $\HL^i(\overline{M})$.
	This latter fact together with \autoref{eq_ex_seq_HL_tM} yields the following isomorphism
	\begin{equation*}
			\HL^i(\overline{M}/t\,\overline{M}) \; \cong \; \HL^i(\overline{M}) \big/ t\, \HL^i(\overline{M}).
	\end{equation*}
	Since $\overline{M}/t\,\overline{M} \cong M \otimes_{B} C/(t)$, the induction hypothesis implies that every $\HL^i(\overline{M}/t\,\overline{M})$ is a free $C/(t)$-module, and so \autoref{lem_base_ch_loc} gives the following isomorphisms 
	\begin{align}
		\label{eq_distribute_filtration}
		\begin{split}
			\HL^i(\overline{M}/I_j \overline{M}) & \; \cong \; \HL^i\big(\overline{M}/t\,\overline{M} \otimes_{C/(t)} C/I_{j}\big) \\
			& \; \cong \; \HL^i(\overline{M}/t\,\overline{M}) \,\otimes_{C/(t)}\, C/I_{j} \\
			& \; \cong \; \HL^i(\overline{M}) \big/ t\, \HL^i(\overline{M}) \,\otimes_{C/(t)}\, C/I_{j} \\
			& \; \cong \; \HL^i(\overline{M}) \big/ I_j\, \HL^i(\overline{M})
			\end{split}
	\end{align}
	for all $0 \le j \le \ell-1$.
	Consequently, the equations \autoref{eq_ex_seq_HL_tM} and \autoref{eq_distribute_filtration} give the equality $\HL^i(I_j\,\overline{M}) = I_j\,\HL^i(\overline{M})$ as submodules of $\HL^i(\overline{M})$.
	As for \autoref{eq_ex_seq_HL_tM}, now the short exact sequence $0 \rightarrow I_1\overline{M} \rightarrow \overline{M} \xrightarrow{\cdot t} t\overline{M} \rightarrow 0$  gives the following short exact sequence 
	$$
	0 \rightarrow I_1\, \HL^i(\overline{M}) \rightarrow \HL^i(\overline{M}) \xrightarrow{\cdot t} t\,\HL^i(\overline{M}) \rightarrow 0,
	$$
	and so from the isomorphism $C/I_1 \cong (t)$ we obtain a natural isomorphism 
	\begin{equation}
		\label{eq_isom_tensor_t}
		\HL^i(\overline{M}) \otimes_C (t)   \xrightarrow{\cong\,} t\,\HL^i(\overline{M}).
	\end{equation}
	Finally, the fact that $\HL^i(\overline{M})/t\, \HL^i(\overline{M})$ is a free $C/(t)$-module, the isomorphism of \autoref{eq_isom_tensor_t}, and the local flatness criterion (see \cite[Theorem 22.3]{MATSUMURA}) imply that $\HL^i(\overline{M})$ is a free $C$-module. 
	So, the proof of the proposition is complete.
\end{proof}

Finally, we are now ready for the main result of this paper.

\begin{theorem}
	\label{thm_freeness_criterion}
	Assume \autoref{setup_local_criterion}. 
	Let $M$ be a finitely generated graded $R$-module and suppose that $M$ is a free $B$-module.
	Then, the following six conditions are equivalent:
	\begin{enumerate}[\rm (1)]
		\item $\HL^i(M)$ is a free $B$-module for all $0 \le i \le r$.
		\item $\Ext_T^i(M, T)$ is a free $B$-module for all $0 \le i \le r$.
		\item $\HL^i(M/\bb^q M)$ is a free $B_q$-module for all $0 \le i \le r$ and $q \ge 1$.
		\item $\Ext_{T_q}^i(M/\bb^q M,  T_q)$ is a free $B_q$-module for all $0 \le i \le r$ and $q \ge 1$.
		\item The natural map 
		$
		\HL^i\big(M/\bb^qM\big) \rightarrow \HL^i\big(M/\bb M\big)
		$ 
		is surjective for all $0 \le i \le r$ and  $q \ge 1$.
		\item The natural map 
		$
		\Ext_{T_q}^{i}(M/\bb M, \omega_{T_q})
		\rightarrow 
		\Ext_{T_q}^{i}(M/\bb^qM, \omega_{T_q})
		$ 
		is injective for all $0 \le i \le r$ and $q \ge 1$.
	\end{enumerate}
	Moreover, when any of the above equivalent conditions is satisfied, we have the following isomorphisms 
	\begin{enumerate}[\rm (i)]
		\item $\HL^i(M) \; \cong \; \big(\Ext_T^{r-i}(M, T(-\delta))\big)^{*_B}$,
		\item $\HL^i(M) \otimes_B C \; \xrightarrow{\cong} \; \HL^i(M \otimes_B 
		C)$, and
		\item $\Ext_T^i(M, T) \otimes_B C \;\xrightarrow{\cong}\; \Ext_{T \otimes_B C}^i(M\otimes_B C, T\otimes_B C)$
	\end{enumerate}	
	for all integers $i$ and any $B$-algebra $C$.
\end{theorem}
\begin{proof}
	The equivalences $(1) \Leftrightarrow (2)$ and $(3) \Leftrightarrow (4)$ are obtained from \autoref{prop_equiv_freeness}.
	The equivalence $(5) \Leftrightarrow (6)$ is a consequence of \autoref{lem_fiber_full_def_Ext}.
	The implication $(1) \Rightarrow (3)$  follows from \autoref{lem_base_ch_loc}.
	The implication $(3) \Rightarrow (5)$ is obtained from \autoref{lem_base_ch_loc}, and \autoref{prop_free_impl_fiber_full} yields the implication $(5) \Rightarrow (3)$. 
	The additional statements regarding the isomorphisms (i), (ii) and (iii) would follow from \autoref{prop_equiv_freeness},  \autoref{lem_base_ch_loc} and \autoref{lem_Ext_base_change}, respectively.
	
	Therefore, to conclude the proof of the theorem it suffices to show the implication $(3) \Rightarrow (1)$. 
	Suppose that $\HL^i(M / \bb^q M)$ is a free $B_q$-module for all $i \ge 0$ and  $q \ge 1$.
	By \autoref{lem_base_ch_loc} we have the natural isomorphism
	$$
	\HL^i(M / \bb^{q+1} M) \otimes_{B_{q+1}} B_q \; \xrightarrow{\cong} \; \HL^i(M / \bb^q M).
 	$$
 	This implies that the natural map $\HL^i(M / \bb^{q+1} M)  \rightarrow \HL^i(M / \bb^q M)$ is surjective.
 	Then, the conditions of \cite[\href{https://stacks.math.columbia.edu/tag/0912}{Tag 0912}]{stacks-project} are satisfied for the inverse system ${\left({[\HL^i(M / \bb^q M)]}_\nu\right)}_{q \ge 1}$, and we get that the inverse limit $\underset{\leftarrow}{\lim}\, {[\HL^i(M / \bb^q M)]}_\nu$ is a flat $B$-module.
 	As a consequence of \autoref{prop_formal_funct}, we have that
 	$
 	{[\HL^i(M)]}_\nu\,{\widehat{}}  \;\xrightarrow{\cong} \; \underset{\leftarrow}{\lim}\, {[\HL^i(M/\bb^qM)]}_\nu
 	$
 	is a flat $B$-module.
 	Finally, this implies that ${[\HL^i(M)]}_\nu$ is a free $B$-module for each $\nu \in \ZZ$, and so the proof of the theorem is complete.
\end{proof}


\section{Fiber-full locus of a module}
\label{sect_fib_full}

In this section, we provide some  applications that follow from \autoref{thm_freeness_criterion}.

Let $A$ be a Noetherian ring, $R$ be a finitely generated positively graded $A$-algebra, and  $\mm = [R]_+$ be the graded irrelevant ideal of $R$.  	
For any $\pp \in \Spec(A)$, let $k(\pp) := A_\pp/\pp A_\pp$ be the residue field of $\pp$.

\begin{definition}
	Let $M$ be a finitely generated graded $R$-module.
	The \emph{fiber-full locus} of $M$ is defined as 
	$$
	\mathfrak{Fib}(M) \; := \; \big\{ \pp \in \Spec(A) \mid M \otimes_{A} A_\pp \text{ is fiber-full over } A_\pp \big\}
	$$
	(see \autoref{nota_fiber_full}).
\end{definition}

The following theorem gives some properties of the fiber-full locus of a module. 

\begin{theorem}
	\label{thm_locus}
	Let $M$ be a finitely generated graded $R$-module.
	Then, the following statements hold:
	\begin{enumerate}[\rm (i)]
		\item $\mathfrak{Fib}(M)$ is an open subset of $\Spec(A)$.
		\item If $A$ is generically reduced, then there is an element $a \in A$ avoiding the minimal primes of $A$ such that $\HL^i(M) \otimes_{A} A_a$ is a projective $A_a$-module for all $i \ge 0$, and so $\mathfrak{Fib}(M)$ is a dense subset of $\Spec(A)$.
		\item For all $i \ge 0$ and $\nu \in \ZZ$ the function
		$$
		\Spec(A) \rightarrow \NN, \quad \pp \;\mapsto\; \dim_{k(\pp)}\left( \left[\HL^i(M \otimes_{A} k(\pp))\right]_\nu \right)
		$$		
		is locally constant on $\mathfrak{Fib}(M)$.
	\end{enumerate}
\end{theorem}
\begin{proof}
	(i) 
	It is a known result that for a finitely generated graded $R$-module $L$, the following locus 
	$$
	U_L \; := \; \big\{ \pp \in \Spec(A) \mid  L \otimes_{A} A_\pp \text{ is a free } A_\pp\text{-module} \big\}
	$$
	is an open subset of $\Spec(A)$ (see, e.g., \cite[Lemma 2.5, Notation 2.6]{GEN_FREENESS_LOC_COHOM}).
	Similarly to \autoref{setup_local_criterion}, we choose a positively graded polynomial ring $T = A[x_1,\ldots,x_r]$ over $A$ such that we have a homogeneous surjection $T \twoheadrightarrow R$.
	Then by \autoref{thm_freeness_criterion} we obtain the equality
	$$
	\mathfrak{Fib}(M) \; = \; U_M \, \cap \, U_{\Hom_T(M, T)} \,\cap\, U_{\Ext_T^1(M, T)} \,\cap\, \cdots \,\cap\, U_{\Ext_T^r(M, T)},
	$$
	and so the result follows.
	
	(ii)
	Suppose that $A$ is generically reduced.
	Let $\{\pp_1,\ldots,\pp_k\}$ be the minimal primes of $A$.
	By part (i), we choose an ideal $I \subset A$ such that $\mathfrak{Fib}(M) = \Spec(A) \setminus V(I)$. 
	As $A$ is generically reduced, $A_{\pp_i}$ is a field for any minimal prime $\pp_i$, and so it is clear that $M \otimes_{A} A_{\pp_i}$ is fiber-full over $A_{\pp_i}$. 
	This implies that $\pp_i \not\supseteq I$.
	Finally, the prime avoidance lemma gives an element $a \in A$ such that $a \in I \setminus (\pp_1 \cup \cdots \cup \pp_k)$.

	(iii)
	Fix any $i \ge 0$ and $\nu \in \ZZ$.
	Let $\pp \in \mathfrak{Fib}(M)$.
	Since $\mathfrak{Fib}(M) \subset \Spec(A)$ is an open subset by part (i), we can choose $0 \neq a \in A$ such that $\pp \in \Spec(A_a) \subset \mathfrak{Fib}(M)$.
	Then, by using \autoref{thm_freeness_criterion}, we obtain that $[\HL^i(M)]_\nu  \otimes_{A} A_a$ is a finitely generated locally free $A_a$-module, and the base change isomorphism 
	$$
	\HL^i(M \otimes_A k(\qqq)) \;\cong\; \HL^i(M) \otimes_A k(\qqq) 
	$$
	for all $\qqq \in \Spec(A_a)$.
	By invoking \cite[\href{https://stacks.math.columbia.edu/tag/00NX}{Tag 00NX}]{stacks-project} there exists an open subset $V \subset \Spec(A_a)$ containing $\pp$ such that the function
	$$
	\Spec(A) \rightarrow \NN, \quad \qqq \;\mapsto\; \dim_{k(\qqq)}\left( \left[\HL^i(M)\right]_\nu \otimes_{A} k(\qqq) \right)
	$$
	is constant on $V$.
	Therefore, the proof is complete.
\end{proof}

\section{Square-free Gr\"obner degenerations \`a la Conca-Varbaro}
\label{sect_grob_degen}

In this section, we give our alternative proof of the main result of \cite{CONCA_VARBARO}.
Another alternative proof of this result was given in \cite{CONFERENCE_LEVICO}.

Let $\kk$ be a field, $S = \kk[x_1,\ldots,x_r]$ be a positively graded polynomial ring and $\mm = (x_1,\ldots,x_r) \subset S$.
Let $A = \kk[t]$ be a polynomial ring.
Let $R = A \otimes_\kk S \cong S[t]$ be a polynomial with grading induced from $S$; i.e., $x_i \in R$ maintains the degree as an element of $S$ and $\deg(t) = 0$. 
Let $<$	be a term order on $S$.

Our proof of the following result is a simple consequence of our previous developments.

\begin{theorem}[Conca-Varbaro]
	\label{thm_grobner_deform}	
	Let $I \subset S$ be a homogeneous ideal.
	If the initial ideal $\iniTerm_<(I)$ is square-free, then 
	$$
	\dim_\kk\left(\left[\HL^i(S/I)\right]_\nu\right) \; = \; \dim_\kk\left(\left[\HL^i(S/\iniTerm_<(I))\right]_\nu\right)
	$$
	for all $i \ge 0 $ and $\nu \in \ZZ$.
\end{theorem}
\begin{proof}
		We can choose a weight vector $\omega \in \NN^{r}$ such that the $\omega$-homogenization $J = \hom_\omega(I) \subset R$ satisfies the following conditions:
		\begin{enumerate}[(i)]
			\item $J$ is an $R$-homogeneous ideal,
			\item $R/J$ is a free $A$-module,
			\item $R/J \otimes_{A} A/(t) \cong S / \iniTerm_<(I)$ and,
			\item $R/J \otimes_{A} \kk(t) \cong S / I \otimes_{A} \kk(t)$
		\end{enumerate}
		(see, e.g., \cite[\S 15.8]{EISEN_COMM}).
		The residue fields of the prime ideals $(0) \subset A$ and $(t) \subset A$ are given by $k\big((0)\big) = \kk(t)$ and $k\big((t)\big) = \kk$, respectively.
		We have that $R/J \otimes_{A} A_{(t)}$ is fiber-full over $A_{(t)}$ by \cite[Proposition 2.3]{CONCA_VARBARO} or \cite[Corollary 3.4]{CONFERENCE_LEVICO}.
		Finally, \autoref{thm_locus} implies that 
		\begin{align*}
			\dim_\kk\left(\left[\HL^i(S/I)\right]_\nu\right)  & \; = \;  \dim_{\kk(t)}\left(\left[\HL^i\big(R/J \otimes_{A} \kk(t)\big)\right]_\nu\right) \\
			& \; = \;  \dim_{\kk}\left(\left[\HL^i\big(R/J \otimes_{A} A/(t)\big)\right]_\nu\right) \\
			 & \; = \;
			  \dim_\kk\left(\left[\HL^i(S/\iniTerm_<(I))\right]_\nu\right)
		\end{align*}
		for all $i \ge 0 $ and $\nu \in \ZZ$.
\end{proof}

For a finitely generated graded $S$-module $M$, we denote the Castelnuovo-Mumford regularity and the depth of $M$ by $\reg(M)$ and $\depth(M)$, respectively, and a non-zero Betti number $\beta_{i,i+j}(M) := \dim_{\kk}\big(\left[\Tor_i^S(\kk, M)\right]_{i+j}\big)$ is called extremal if $\beta_{h,h+k}(M)=0$ for every $h \ge i, k \ge j$ with $(h,k) \neq (i,j)$.

\begin{corollary}
	Let $I \subset S$ be a homogeneous ideal.
	If $\iniTerm_<(I)$ is square-free, then $S/I$ and $S/\iniTerm_<(I)$ have the same extremal Betti numbers, and, in particular, 
	$$
	\depth(S/I) = \depth(S/\iniTerm_<(I)) \quad \text{ and } \quad \reg(S/I) = \reg(S/\iniTerm_<(I)).
	$$
\end{corollary}
\begin{proof}
	It follows from \autoref{thm_grobner_deform} and \cite[Theorem 3.11]{SCHENZEL_NOTES}.
\end{proof}


\end{document}